\documentclass{article}
\usepackage{amsmath,amssymb,amsthm,amscd,dsfont}
\numberwithin{equation}{section} \allowdisplaybreaks
\begin{document}
\newtheorem{theorem}{Theorem}[section]
\newtheorem{defin}{Definition}[section]
\newtheorem{prop}{Proposition}[section]
\newtheorem{corol}{Corollary}[section]
\newtheorem{lemma}{Lemma}[section]
\newtheorem{rem}{Remark}[section]
\newtheorem{example}{Example}[section]
\title{A note on submanifolds and mappings in generalized complex geometry}
\author{{\small by}\vspace{2mm}\\Izu Vaisman}
\date{}
\maketitle
{\def\thefootnote{*}\footnotetext[1]%
{{\it 2000 Mathematics Subject Classification: 53C56} .
}}
\begin{center} \begin{minipage}{12cm}
A{\footnotesize BSTRACT. In generalized complex geometry, we revisit linear subspaces and submanifolds that have an induced generalized complex structure. We give an expression of the induced structure that allows us to deduce a smoothness criterion, we dualize the results to submersions and we make a few comments on generalized complex mappings. Then, we discuss submanifolds of generalized K\"ahler manifolds that have an induced generalized K\"ahler structure. These turn out to be the common invariant submanifolds of the two classical complex structures of the generalized K\"ahler manifold.}
\end{minipage}
\end{center} \vspace{5mm}
\section{Introduction}
\label{intro}
Induced generalized complex structures of submanifolds were introduced and studied in \cite{BB} and the study was continued in \cite{{BS},{Goto},{IV}}, etc. After a preliminary Section 2, where we recall known results, in Section 3, we give a general expression of the induced structure and derive a smoothness criterion. The results are dualized to submersions and a few remarks on more general mappings are made, including a proposed, new definition of generalized complex mappings. In Section 4, we consider the generalized Hermitian and generalized K\"ahler case, which is known to be given by a 1-1 correspondence with quadruples $(\gamma,\psi,J_\pm)$, where $\gamma$ is a metric, $\psi$ is a $2$-form and $J_\pm$ are classical $\gamma$-compatible complex structures. Then, we define the notion of induced structure in a way that is compatible with the generalized complex case and is such that existence of the induced structure is equivalent to $J_\pm$-invariance. We work in the $C^\infty$-category and use the classical notation of Differential Geometry \cite{KN}.
\section{Preliminaries}
In this section, required by the referee for the benefit of the non-expert reader, we recall some basic notions of generalized complex geometry. More details may be found in \cite{Galt}.

Let $M$ be an $m$-dimensional, differentiable manifold. Its {\it big tangent bundle} is the Whitney sum of the tangent and cotangent bundle: $\mathbf{T}M=TM\oplus T^*M$,
which has the
non degenerate, neutral metric
\begin{equation}\label{gptCourant} g((X,\alpha),(Y,\mu))=
\frac{1}{2}(\alpha(Y)+\mu(X)),
\end{equation}
where $X,Y\in
\chi^1(N),\,\alpha,\mu\in\Omega^1(N)$ ($\chi^k(M)$ is
the space of $k$-vector fields and $\Omega^k(M)$ is the space of
differential $k$-forms on $M$; we will also use calligraphic
characters for pairs: $\mathcal{X}=(X,\alpha),\mathcal{Y}=(Y,\mu)$,
etc.).
The metric reduces the structure group of $\mathbf{T}M$ to $O(m,m)$. The scope of generalized geometry is the study of
{\it generalized geometric structures}, defined as further reductions of the structure group
$O(m,m)$ to various subgroups.

If only linear algebra is concerned (which is the situation of the big tangent bundle at a point of $M$), we may look at a finite-dimensional vector space $W$ and at $\mathbf{W}=W\oplus W^*$, endowed with the similar metric (\ref{gptCourant}). In the whole paper, we will use the same notation for vectors and tensors in the linear and non-linear case and the context will show in which case we are.

In the case of a manifold $M$, there exists another important operation defined on smooth sections of $\mathbf{T}M$,
the {\it Courant bracket}, given by
\begin{equation}\label{crosetC}
[(X,\alpha),(Y,\mu)]=([X,Y],L_X\mu-L_Y\alpha
+\frac{1}{2}d(\alpha(Y)-\mu(X)),\end{equation}
where $[X,Y]$ is the Lie bracket and $L$ is the Lie derivative.

A $B$-field (gauge) transformation defined by a $2$-form $B$ of $M$ is a mapping
$$(X,\alpha)\mapsto(X,\alpha+i(X)B).$$ The transformation preserves the metric and, if the form $B$ is closed, the transformation preserves the Dirac bracket too.

A real, respectively complex, {\it almost Dirac structure} on $M$ is a maximally $g$-isotropic subbundle $D$ of $\mathbf{T}M$, respectively $\mathbf{T}^cM=\mathbf{T}M\otimes\mathds{C}$ (the upper index $c$ denotes complexification). Then, if $D$ is closed under Dirac brackets, $D$ is said to be an {\it integrable} or {\it Dirac structure}. The $B$-field transformations send almost Dirac structures to almost Dirac structures and, if $B$ is closed, send Dirac structures to Dirac structures.

A {\it generalized, almost complex structure} is a reduction of the structure group to $O(m,m)\cap Gl(2m,\mathds{C})$, equivalently, it is an endomorphism $\mathcal{J}\in
End(\mathbf{T}M)$ such that
\begin{equation}\label{skewsym}	 \mathcal{J}^2=-
Id,\;\;g(\mathcal{X},\mathcal{J}\mathcal{Y}) +
g(\mathcal{J}\mathcal{X},\mathcal{Y})=0.\end{equation}
Like in the classical case, $\mathcal{J}$ is equivalent with a decomposition
$\mathbf{T}^cM=L\oplus\bar L$, where $L,\bar{L}$ are the $\pm i$-eigenbundles of $\mathcal{J}$ and (\ref{skewsym}) ensures that $L$ and its complex conjugate $\bar L$ are complex almost Dirac structures.

The expression
\begin{equation}\label{NijPsi}
\mathcal{N}_\mathcal{J}(\mathcal{X},\mathcal{Y}) = [\mathcal{J}\mathcal{X},
\mathcal{J}\mathcal{Y}] -\mathcal{J}[\mathcal{X},\mathcal{J}\mathcal{Y}] -
\mathcal{J}[\mathcal{J}\mathcal{X},\mathcal{Y}]
-[\mathcal{X},\mathcal{Y}]
\end{equation}
has a tensorial character and it is called the {\it Courant-Nijenhuis
torsion} of $\mathcal{J}$. If $
\mathcal{N}_\mathcal{J}=0$, the structure is {\it integrable} and the term
``almost" is omitted. Integrability is equivalent with the closure of $L$ under Courant brackets. Hence, on a generalized complex manifold $L,\bar{L}$ are complex Dirac structures.

In other words, a generalized almost complex structure is equivalent with a maximally $g$-isotropic (almost Dirac) subbundle $L\subset\mathbf{B}^c$ such that $L\cap\bar{L}=0$ (bar denotes complex conjugation) and the structure is integrable iff $L$ is closed under the Courant bracket.

A manifold endowed with an (almost) generalized complex structure is called an (almost) generalized complex manifold.

In terms of classical tensor fields an almost generalized complex structure may be written as
\begin{equation}\label{J} \mathcal{J}\left(
\begin{array}{c}X\vspace{2mm}\\ \alpha \end{array}
\right) = \left(\begin{array}{cc} A&\sharp_\pi\vspace{2mm}\\
\flat_\sigma&-A^*\end{array}\right) \left( \begin{array}{c}X\vspace{2mm}\\
\alpha \end{array}\right),\end{equation}
where $*$ denotes transposition, $A\in End(TM),\pi\in\chi^2(M),\sigma\in\Omega^2(M)$, $\sharp_\pi\alpha=i(\alpha)\pi,\flat_\sigma X=i(X)\sigma$. The conditions satisfied by $\mathcal{J}$ are equivalent to
\begin{equation}\label{2-reddezv} A^2+\sharp_\pi\circ\flat_\sigma=-Id,\;\sharp_\pi\circ A^*=A\circ\sharp_\pi,\;\flat_\sigma\circ A=A^*\circ\flat_\sigma\end{equation}
and imply the fact that $m=dim\,M$ must be even.  With (\ref{J}), we have
\begin{equation}\label{eqluiL} L=\{(X-i(AX+\sharp_\pi\alpha),
\alpha-i(\flat_\sigma X-\alpha\circ A))\}.
\end{equation}

If only algebraic facts are concerned (no Dirac brackets or integrability involved), the notions and notation above apply to {\it linear, generalized complex structures} $\mathcal{J}$ of a real vector space $W$.

A Euclidean metric $G$ of a space $\mathbf{W}=W\oplus W^*$ is a {\it generalized metric} of $W$ if there exists a decomposition $\mathbf{W}=W_+\oplus W_-$, where $W_\pm$ are maximal $g$-positive and $g$-negative subspaces (thus, $dim\,W_\pm=dim\,W$), such that $W_+\perp_gW_-,\,W_+\perp_GW_-,\,G|_{W_\pm}=\pm2g|_{W_\pm}$.
Such a metric is equivalent with a $g$-symmetric product structure $H\in End(\mathbf{W})$, $H^2=Id.$ The subspaces $W_\pm$ are the $\pm1$-eigenspaces of $H$ (thus, $H$ is a generalized paracomplex structure \cite{IV}) and
$$ G((X,\alpha),(Y,\beta))= 2g(H(X,\alpha),(Y,\beta)).$$

Since the elements of $W_\pm$ are not $g$-isotropic, we have $W_\pm\cap W=0$, hence the projections $pr_{W}:W_{\pm}\rightarrow W$ are isomorphisms with the inverse denoted by $\tau_\pm:W\rightarrow W_{\pm}$ and $\tau_\pm X=(X,\flat_{\theta_\pm} X)$ ($X\in W$) where $\theta_{\pm}$ are non degenerate, $2$-covariant tensors of $W$.
Furthermore, $g(\tau_+X,\tau_-X')=0$ implies $\theta_+(X,X')=-\theta_-(X',X)$ and the two tensors must have the same skew-symmetric part and opposite symmetric parts. I.e., $\theta_\pm=\psi\pm\gamma$ where $\psi\in\wedge^2W^*$ and $$\gamma(X,X')=g(\tau_\pm X,\tau_\pm X') =(1/2)G(\tau_\pm X,\tau_\pm X')$$ is a positive definite metric on $W$. This shows that generalized metrics $G$ bijectively correspond to pairs $(\gamma,\psi)$ \cite{Galt}.

The formula
$$ (X,\alpha)=(X_+,\flat_{\psi+\gamma}X_+)+ (X_-,\flat_{\psi-\gamma}X_-),\;X_\pm= \frac{1}{2}((Id\pm\varphi)X\pm\sharp_\gamma\alpha),$$
where $\varphi=-\sharp_\gamma\flat_\psi$, yields the $W_\pm$ decomposition of $(X,\alpha)$. Since $W_\pm$ are the $\pm1$-eigenbundles of $H$, we deduce the following matrix representation
\begin{equation}\label{matriceG}
H\left(
\begin{array}{c}X\vspace{2mm}\\ \alpha \end{array}
\right) = \left(\begin{array}{cc} \varphi&\sharp_\gamma\vspace{2mm}\\
\flat_\beta&\varphi^*\end{array}\right) \left( \begin{array}{c}X\vspace{2mm}\\
\alpha \end{array}\right), \end{equation}
where $\flat_\beta=\flat_\gamma\circ(Id-\varphi^2)$ and $\beta$ is also symmetric and positive definite.

A {\it generalized Hermitian structure} of $(W,G)$ is a generalized complex structure $\mathcal{J}$ that is $G$-skew-symmetric, equivalently, commutes with $H$. If this happens, $\mathcal{J}'=H\circ\mathcal{J}$ is a second generalized Hermitian structure of $(W,G)$, which commutes with $\mathcal{J}$ and $H=-\mathcal{J}\circ\mathcal{J}'$.
If $\mathcal{J}$ is given, $W^c_\pm=L_\pm\oplus\bar{L}_\pm$, where the terms are $\pm i$-eigenspaces of the complex structures $\mathcal{J}_\pm$ induced by $\mathcal{J}$ in $W_\pm$ \cite{Galt}. It follows that the generalized Hermitian structure is equivalent with a decomposition (see also \cite{BS})
\begin{equation}\label{descptKgen}\mathbf{W}^c=L_+\oplus L_-\oplus \bar{L}_+\oplus\bar{L}_-\end{equation}
such that: (i) $L_+\oplus L_-, L_+\oplus \bar{L}_-$ are Dirac subspaces, (ii) $L_+\oplus \bar{L}_+, L_-\oplus \bar{L}_-$ are the complexification of the $\pm1$-eigenspaces of a product structure $H$ associated to a generalized Euclidean metric $G$. From (i), (ii) it follows that the terms of the decomposition are $g$-isotropic subspaces of equal dimensions.
The spaces in condition (i) define the generalized complex structures $\mathcal{J},\mathcal{J}'$.
A space $W$ endowed with a generalized Hermitian structure $(G,\mathcal{J})$ is a {\it generalized Hermitian space}.

Here are a few other important facts \cite{Galt}. The projections of the structures $\mathcal{J}_\pm$ to $W$ are complex structures $J_\pm$ such that $(\gamma,J_\pm)$, where $G$ corresponds to the pair $(\gamma,\psi)$, are classical Hermitian structures. Thus, the generalized Hermitian structures are in a bijective correspondence with quadruples $(\gamma,\psi,J_\pm)$.
The structure $(G,\mathcal{J}')$ corresponds to the quadruple
$(\gamma,\psi,J_+,-J_-)$.

Furthermore \cite{VgenS}, if $\mathcal{J}$ is given by (\ref{J}), we get
\begin{equation}\label{JApi}
J_\pm=pr_{W_\pm}\circ\mathcal{J}_\pm\circ\tau_\pm =A+\sharp_\pi\circ\flat_{\psi\pm\gamma},\end{equation}
which yields the classical tensors of a generalized Hermitian structure:
\begin{equation}\label{clasicK}
\begin{array}{l} \sharp_\pi=\frac{1}{2}(J_+-J_-)\circ\sharp_\gamma,\; \sharp_{\pi'}=\frac{1}{2}(J_++J_-)\circ\sharp_\gamma,\vspace*{2mm}\\ A=\frac{1}{2}[J_+\circ(Id+\varphi)+J_-\circ(Id-\varphi)],\vspace*{2mm}\\
A'=\frac{1}{2}[J_+\circ(Id+\varphi)-J_-\circ(Id-\varphi)],\vspace*{2mm}\\  \flat_\sigma=\flat_\gamma\circ(A\circ\varphi-\varphi\circ A+\sharp_\pi\circ\flat_\beta),\vspace*{2mm}\\ \flat_{\sigma'}=\flat_\gamma\circ(A'\circ\varphi-\varphi\circ A'+\sharp_{\pi'}\circ\flat_\beta).\end{array}\end{equation}
\begin{example}\label{exHlin} {\rm We may give a classical, Hermitian structure $(\gamma,J_+)$ on a space $W^{2n}$ by an adapted frame in $W^c$, i.e., line matrices of complex vectors $e=(e_1,...,e_{n}),\bar{e}=(\bar{e}_1,...,\bar{e}_{n})$ such that
\begin{equation}\label{auxnon}J_+e=ie,\, \gamma(e_i,e_j)=\gamma(\bar{e}_i,\bar{e}_j)=0,\,\gamma(e_i,\bar{e}_j)=\delta_{ij},
\end{equation}
where $J_+,\gamma$ are extended to $W^c$ by complex linearity. A corresponding, real, $\gamma$-orthonormal basis of $W$ is given by
\begin{equation}\label{realbasis} w_i=\frac{1}{\sqrt{2}}(e_i+\bar{e}_i),\,J_+w_i=\frac{i}{\sqrt{2}}(e_i-\bar{e}_i),\;
i=1,...,2n.\end{equation}
Then, a second $\gamma$-Hermitian structure $J_-$ will be determined by a frame $(f=e\Theta+\bar{e}\Psi,\bar{f}=\bar{e}\bar{\Theta}+e\bar{\Psi})$, where $\Theta,\Psi$ are complex $(n,n)$-matrices such that the following conditions hold (a)
$rank\left(\begin{array}{cc}\Theta,\bar{\Psi}\\ \Psi,\bar{\Theta}\end{array}\right)=2n$, (b) $\gamma(f_i,f_j)=0$. Condition (a) shows linear independence and the equality $span\{f\}\cap span\{\bar{f}\}=0$, therefore, $J_-f=if$ defines a complex structure on $W$. Condition (b) is equivalent to $\Theta^*\Psi+\Psi^*\Theta=0$ and ensures the compatibility property $\gamma(J_-X,J_-Y)=\gamma(X,Y)$, $\forall X,Y\in W$. The quadruple $(\gamma,\psi,J_+,J_-)$, where $\psi\in\wedge^2W^*$ defines a linear generalized Hermitian structure on $W$.}\end{example}

Above, we have referred to linear generalized Hermitian structures. Moving over to manifolds $M$, we get the corresponding notion of an {\it almost generalized Hermitian structure} that will make $M$ an {\it almost generalized Hermitian manifold}.

If the two structures $\mathcal{J},\mathcal{J}'$ are integrable, the almost generalized Hermitian structure is called a {\it generalized K\"ahler structure} and it makes $M$ a {\it generalized K\"ahler manifold}. The double integrability condition is equivalent to the closure of the spaces $\Gamma L_\pm$ under the Dirac bracket \cite{Galt} ($\Gamma$ denotes the space of smooth cross sections of a vector bundle). Furthermore, $M$ is a generalized K\"ahler manifold iff the classical structures $J_\pm$ are integrable and
\begin{equation}\label{KalercuLC} (\nabla_XJ_\pm)Y=\pm\frac{1}{2}\sharp_\gamma[i(X)i(J_\pm Y)d\psi
-(i(Y)i(X)d\psi)\circ J_\pm],\end{equation}
where $\nabla$ is the Levi-Civita connection of the metric $\gamma$ of the quadruple $(\gamma,\psi,J_\pm)$ of $M$ \cite{{Galt},{VgenS}}.
\section{Submanifolds with induced generalized complex structure}
The basic notion required for the discussion of induced generalized complex structures is that of a pullback space. This notion is defined as follows. Let $V,W$ be finite-dimensional linear spaces and put $\mathbf{V}=V\oplus V^*,\mathbf{W}=W\oplus W^*$. Consider a linear mapping $f_*:V\rightarrow W$ with transposed mapping $f^*:W^*\rightarrow V^*$. If $U$ is a subspace of $\mathbf{W}$, its {\it pullback} is
\begin{equation}\label{pullbackgeneral}
\overleftarrow{f}^*U=\{(X,f^*\eta)\,/\,(f_*X,\eta)\in U,\;X\in V,\eta\in W^*\}\subseteq\mathbf{V}.\end{equation}
(We added an arrow to the usual notation in order to avoid confusion with set-theoretic image.)
In particular, the pullback of a Dirac subspace of $\mathbf{W}$ is a Dirac subspace of $\mathbf{V}$ (e.g., \cite{BB}).

Now, we define the notion of an induced structure.
\begin{defin}\label{defsbvarc} {\rm\cite{BB} Let $W$ be a linear space endowed with a linear generalized complex structure $\mathcal{J}$ with the $i$-eigenspace $L$. A subspace $\iota_*:V\subseteq W$ is a {\it subspace with induced structure} or {\it generalized complex subspace} if the pullback space $\tilde{L}=\overleftarrow{\iota}^*L$ defines a linear generalized complex structure $\tilde{\mathcal{J}}$ on $V$. $\tilde{\mathcal{J}}$ will be called the {\it induced structure}. A submanifold $\iota:N\hookrightarrow M$ of an almost generalized complex manifold $(M,\mathcal{J})$ is a {\it submanifold with induced structure} $\tilde{\mathcal{J}}$ ({\it generalized almost complex submanifold}) if, $\forall x\in N$, $\iota_{*x}:T_xN\subseteq (T_{\iota(x)},\mathcal{J}_x)$ is a subspace with induced structure $\tilde{\mathcal{J}}_x$ and the field of induced structures is a smooth bundle endomorphism $\tilde{\mathcal{J}}$ of $\mathbf{T}N$.}\end{defin}

The existence of the induced structure is equivalent to $\overleftarrow{\iota}^*L\cap{\overleftarrow{\iota}^*\bar{L}}=0$ and in the case of submanifolds one must add the smoothness condition. From (\ref{pullbackgeneral}) and (\ref{eqluiL}) we get
\begin{equation}\label{eqluiL1} \begin{array}{l}
\overleftarrow{\iota}^*L=\{(X+iX',\iota^*\eta-i\iota^*(\flat_\sigma \iota_*X-\eta\circ A))
\vspace*{2mm}\\ /\,X,X'\in V,\,\eta\in W^*,\,\iota_*X'=-A\iota_*X-\sharp_\pi\eta\}.
\end{array}\end{equation}

For the subspace $\iota_*:V\subseteq W$ we identify
$$V\approx im\,\iota_*,\, V^*=im\,\iota^*\approx W^*/ker\,\iota^*,\,ker\,\iota^*\approx ann\,V,\,\mathbf{V}\approx B/ann\,V,$$
where $B=V\oplus W^*$. The last quotient exists since
$ker\,\iota^*=B^{\perp_g}\subseteq B$ and the isomorphism is induced by the natural projection $s:B\rightarrow B/ker\,\iota^*$ \cite{BS}.

In \cite{IV} it was shown that the subspace $V\subseteq(W,\mathcal{J})$ has the induced structure iff
\begin{equation}\label{gcsusp} V\cap\sharp_\pi(ann\,V)=0,\, A(V)\subseteq V+ im\,\sharp_\pi\end{equation} and, then, one also has
\begin{equation}\label{condcuoplus} V+ im\,\sharp_\pi=V\oplus\sharp_\pi(ann\,V).\end{equation}
The meaning of the first condition (\ref{gcsusp}) is that $V$ is a {\it Poisson-Dirac subspace} of $W$ in the sense of \cite{CF}. By definition, this means that $\overleftarrow{f}^*(graph\,\sharp_\pi)$ ($\sharp_\pi:W^*\rightarrow W$) is the graph of $\sharp_{\tilde{\pi}}:V^*\rightarrow V$ for some bivector $\tilde{\pi}\in\wedge^2V$. The bivector $\tilde{\pi}$ may be calculated as follows. The first condition (\ref{gcsusp}) is equivalent to
\begin{equation}\label{auxPD} ann\,V+[ann\,\sharp_\pi(ann\,V)]=W^* \hspace{3mm}(ann\,\sharp_\pi(ann\,V)=(ann\,V)^{\perp_\pi}).
\end{equation}
By (\ref{auxPD}), $ann\,V$ has a $\pi$-orthogonal complement $\Pi$ in $W^*$ and $V^*\approx\Pi$, $\tilde{\pi}=\pi|_{\Pi}$. The classical tensors of the induced structure are $\tilde{\pi}, \tilde{A}=pr_V\circ A, \flat_{\tilde{\sigma}}=pr_\Pi\circ\flat_\sigma$ \cite{IV}.

In \cite{BS} it was shown that $V$ is a generalized complex subspace of $(W,\mathcal{J})$ iff one of the following equivalent conditions holds: (a) $B=B\cap\mathcal{J}B+ann\,V$, (b) $\mathcal{J}B\subseteq B+\mathcal{J}(ann\,V)$, (c) $(\mathcal{J}B)^{\perp_g}\cap B\subseteq ann\,V$.
\begin{example}\label{exindcstr} {\rm
Let $V$ be a subspace such that $\mathcal{J}B\subseteq B$, equivalently, since $ann\,V=B^{\perp_g}$, $\mathcal{J}(ann\,V)\subseteq ann\,V$ ($V$ is a {\it totally invariant subspace} or has an {\it invariant conormal space}) \cite{{BB},{Goto}}. Then, $AV\subseteq V,im\,\sharp_\pi\subseteq V$ and (\ref{gcsusp}) holds. For a second example, notice that, if $\pi$ is non degenerate, the first condition (\ref{gcsusp}) suffices for $V$ to be a generalized complex subspace of $(W,\mathcal{J})$. In particular, if $(W,\varpi)$ is a symplectic space and $W=V\oplus V^{\perp_\varpi}$, then, $V$ has a generalized complex structure induced by that of $W$, where $A=0,\sigma=\varpi,\pi=-\varpi^{-1}$, as well as by any $B$-field transform of the former (the action of the $B$-field transformation on $\mathcal{J}$ is defined by $\mathcal{J}\mapsto\mathcal{B}^{-1}\mathcal{J}\mathcal{B}$, where
$\mathcal{B}=\left(\begin{array}{ll} Id&0\\ \flat_B&Id\end{array}\right)$) \cite{{BB},{IV}}. The $B$-field transformation changes $A$ to $A-\varpi^{-1}\circ\flat_B$ and, since it is easy to construct a $2$-form $B$ such that $\varpi^{-1}\circ\flat_B(V)\nsubseteq V$, we have an example where a subspace with induced structure is not $A$-invariant. Finally, assume that the generalized complex space $(W,\mathcal{J})$ is endowed with a $\pi$-preserving linear involution $\phi_*$ ($\phi_*^2=Id$) such that $A(\phi_* X)-\phi_*(AX)\in im\,\sharp_\pi$, $\forall X\in W$. Then, it is easy to check the equalities (\ref{gcsusp}) for the subspace $V=\{X+\phi_*X\,/\,X\in W\}$ of the fixed vectors of $\phi_*$, which, therefore, has the induced generalized complex structure.}\end{example}

It is known that, if $\iota:N\hookrightarrow(M,\mathcal{J})$, where $\mathcal{J}$ is integrable, has the induced structure, the latter is integrable too and we have a {\it generalized complex submanifold} \cite{BB}. On the other hand, we give the following criterion of smoothness.
\begin{prop}\label{sbvardif} Let $N\subseteq (M,\mathcal{J})$ be a submanifold of a generalized almost complex manifold such that, $\forall x\in N$, the subspace $T_xN\subseteq T_xM$ has the induced structure $\tilde{\mathcal{J}}_x$. Then, if $TN+im\,\sharp_\pi$ is a smooth subbundle of the restriction $T_NM$ of the tangent bundle $TM$ to $N$, the field of induced structures is smooth, and, if $\mathcal{J}$ is integrable, $\tilde{\mathcal{J}}$ is integrable too. \end{prop}
\begin{proof} Recall that calligraphic letters denote pairs in $TM\oplus T^*M$. We come back to linear subspaces $V\subseteq(W,\mathcal{J})$ and look for a general expression of the induced structure $\tilde{\mathcal{J}}$. Property (a) mentioned above yields $\mathbf{V}=s(B\cap\mathcal{J}B)$, therefore, for all  $s\mathcal{X}\in\mathbf{V}$ there exists $\mathcal{Y}\in B\cap(\mathcal{J}B)$ (equivalently, $\mathcal{Y}\in B$ and $\mathcal{Y}\in\mathcal{J}B$) such that $s\mathcal{X}=s\mathcal{Y}$. Accordingly, we get
\begin{equation}\label{strcindusa} \begin{array}{l}
\tilde{\mathcal{J}}(s\mathcal{X})=\tilde{\mathcal{J}}(s\mathcal{Y})=\frac{1}{2}
\tilde{\mathcal{J}}[s(\mathcal{Y}-i\mathcal{J}\mathcal{Y}) +s(\mathcal{Y}+i\mathcal{J}\mathcal{Y})]\vspace*{2mm}\\ \hspace*{13mm}=
\frac{i}{2}[s(\mathcal{Y}-i\mathcal{J}\mathcal{Y})-s(\mathcal{Y}+i\mathcal{J}\mathcal{Y})]
=s(\mathcal{J}\mathcal{Y}),
\end{array}\end{equation}
where the third equality holds because the terms belong to the $\pm i$-eigenspaces of $\tilde{\mathcal{J}}$ (check that $\overleftarrow{\iota}^*L=s(L\cap B)$ and use $\mathcal{Y},\mathcal{J}\mathcal{Y}\in B$).

For $\mathcal{X}=(X,\alpha)\in B$, a corresponding $\mathcal{Y}$ may be obtained as follows. In view of the second relation (\ref{gcsusp}) and of (\ref{condcuoplus}), there exists a decomposition $AX=X'+\sharp_\pi\nu$, where $X'\in V,\nu\in ann\,V$. On the other hand, (\ref{auxPD}) yields a decomposition $\alpha=\eta+\theta$, where $\eta\in ann\,V,\theta\in ann\,\sharp_\pi(ann\,V)$, hence, $\sharp_\pi\theta\in V$. Using these decompositions and the expression (\ref{J}), we see that
$$pr_V[\mathcal{J}(X,\theta-\nu)]=AX-\sharp_\pi\nu+\sharp_\pi\theta =X'+\sharp_\pi\theta\in V,$$ therefore, $\mathcal{J}(X,\theta-\nu)\in B$ and $(X,\theta-\nu)\in B\cap(\mathcal{J}B)$. Since $\eta,\nu\in ann\,V$, we have $s(X,\alpha)=s(X,\theta-\nu)$ and it follows that we may take $\mathcal{Y}=(X,\theta-\nu)$ and the induced structure is given by
\begin{equation}\label{strcindusa2} \tilde{\mathcal{J}}(s(X,\alpha))=s(\mathcal{J}(X,\theta-\nu)).\end{equation}

Now, we go to the manifolds $N\subseteq M$ of the proposition. Smoothness of $\tilde{\mathcal{J}}$ means that,
if we apply $\tilde{\mathcal{J}}$ to smooth, local, cross sections of $\mathbf{T}N$ we get smooth cross sections again. We use formula (\ref{strcindusa2}), where $X$ is a smooth, local vector field on $N$ and $\alpha$ is a local $1$-form on $M$. Then, we may assume that the terms $X',\sharp_\pi\nu,\eta,\theta$ of the corresponding decompositions used to get (\ref{strcindusa2}) are smooth. Furthermore, since $ann(im\,\sharp_\pi)=ker\,\sharp_\pi$, if $TN+im\,\sharp_\pi$ is a subbundle of $T_NM$, $$ann(TN+im\,\sharp_\pi)=(ann\,TN)\cap(ker\,\sharp_\pi)$$ is a subbundle of $ann\,TN$ and $\sharp_\pi(ann\,TN)\approx ann\,TN/[(ann\,TN)\cap(ker\,\sharp_\pi)]$. This isomorphism tells us that the smooth local vector field $\sharp_\pi\nu$ may be defined with a smooth form $\nu\in ann\,TN$ and, then, (\ref{strcindusa2}) shows the smoothness of $\tilde{\mathcal{J}}$. The last assertion follows from the observation that preceded the proposition. Q.e.d. \end{proof}
\begin{rem}\label{obsexprJind} {\rm Another expression of $\tilde{\mathcal{J}}$ follows from property (b) which yields decompositions
$$
\mathcal{J}\mathcal{X} =\mathcal{X}'+\mathcal{J}\eta_0,\; \mathcal{X}\in B, \mathcal{X}'\in B,\eta_0\in ann\,V.
$$
Then, $\mathcal{Y}=\mathcal{X}-\eta_0\in B\cap(\mathcal{J}B)$ and $s\mathcal{Y}=s\mathcal{X}$. Therefore, by (\ref{strcindusa}),
$$ \tilde{\mathcal{J}}(s\mathcal{X})=
s(\mathcal{J}(\mathcal{X}-\eta_0)). $$
}\end{rem}

In the remaining part of this section we will discuss other types of mappings.
\begin{defin}\label{goodback} {\rm The linear mapping $f_*:V\rightarrow(W,\mathcal{J})$ is {\it backward regular} if $\overleftarrow{f}^*L\cap \overleftarrow{f}^*\bar{L}=(ker\,f_*)^c$ (which is the smallest possible intersection). The smooth mapping $f:N\rightarrow(M,\mathcal{J})$ is {\it backward regular} if $f_{*x}$ is backward regular for all $x\in N$ and the field $\overleftarrow{f}_x^*L$ is a smooth vector bundle.}\end{defin}
\begin{prop}\label{quasiprop} The mapping $f_*$ is backward regular iff it satisfies the conditions
\begin{equation}\label{1ptquasi} im\,f_*\cap(\sharp_\pi(ker\,f^*))=0,\;
A(im\,f_*)\subseteq im\,f_*+im\,\sharp_\pi.\end{equation}\end{prop}
\begin{proof} Using (\ref{eqluiL1}) with $\iota$ replaced by $f$, we get
\begin{equation}\label{LcapbarL} \begin{array}{l}
\overleftarrow{f}^*L\cap \overleftarrow{f}^*\bar{L}= \{(X+iX',f^*\eta-if^*(\flat_\sigma f_*X-\eta\circ A))
\vspace*{2mm}\\ /\,f^*(\flat_\sigma f_*X-\eta\circ A)=f^*(\nu\circ A),\vspace*{2mm}\\
\,f_*X'=-Af_*X-\sharp_\pi\eta,\vspace*{2mm}\\ f_*X'=Af_*X +\sharp_\pi\eta +2\sharp_\pi\nu,\,\nu\in ker\,f^*\}.
\end{array}\end{equation}
The last two conditions are equivalent to
\begin{equation}\label{LauxbarL2} f_*X'=\sharp_\pi\nu,\,Af_*X+\sharp_\pi\eta=-\sharp_\pi\nu.
\end{equation}

Backward regularity occurs iff the triple of conditions included in (\ref{LcapbarL}) hold only for $f_*X'=f_*X=0,
f^*\eta=0,f^*(\eta\circ A)=0.$ For $X'$, the only condition is that given by the first equality (\ref{LauxbarL2}) and it implies $f_*X'=0$ iff the first condition (\ref{1ptquasi}) holds.
Then, if we apply $A$ to the second condition (\ref{LauxbarL2}) and use (\ref{2-reddezv}) and the first condition in (\ref{LcapbarL}), we see that the first condition (\ref{1ptquasi}) also implies $f_*X=0$. Therefore, after imposing the first condition (\ref{1ptquasi}), what remains from (\ref{LauxbarL2}) is $\sharp_\pi\eta=0$, which, in view of (\ref{2-reddezv}), also implies $\sharp_\pi(\eta\circ A)=0$.
Now, we see that $\eta$ satisfies the requirements $f^*\eta=0,f^*(\eta\circ A)=0$ iff
$$ker\,\sharp_\pi\subseteq ker\,f^*\cap ker(f^*\circ\hspace{1.5mm}^t\hspace{-1.5mm}A).$$
Taking above the $g$-orthogonal spaces we get an equivalent condition that exactly is the second condition (\ref{1ptquasi}). Q.e.d. \end{proof}

Now, we shall refer to the dual case. For the linear surjection $f_*:V\rightarrow W$ (also considered in \cite{BB}), we have $ker\,f^*=0$ and we identify $$W=im\,f_*\approx V/ker\,f_*,\, W^*\approx ann(ker\,f_*)\subseteq V^*,\,\mathbf{W}=B/ker\,f_*;$$ we still have $B=V\oplus W^*$, but, $ker\,f_*\subseteq B$ and there exists a projection	 $q:B\rightarrow B/ker\,f_*$.

We recall the definition of the {\it push-forward} operation applied to a subspace $S\subseteq\mathbf{V}$, namely,
\begin{equation}\label{pushforward}
\overrightarrow{f}_*S=\{(f_*X,\eta)\,/\,(X,f^*\eta)\in S,\;X\in V,\eta\in W^*\}\subseteq\mathbf{W}.
\end{equation}
This operation sends Dirac subspaces of $\mathbf{V}$ to Dirac subspaces of $\mathbf{W}$ (e.g., \cite{BB}).

In the following definition we apply the push-forward operation to generalized complex structures.
\begin{defin}\label{projstr} {\rm Let $(V,\mathcal{J})$ be a generalized complex linear space with the $i$-eigenspace $L$. We say that the surjection $f_*:(V,\mathcal{J})\rightarrow W$ has {\it image with projected structure} if $\overrightarrow{f}_*L\cap\overrightarrow{f}_*\bar{L}=0$, where $L$ is the $i$-eigenspace of $\mathcal{J}$.}\end{defin}

From (\ref{pushforward}), it follows that the surjection $f_*:V\rightarrow W$ has the image with projected structure iff the injection $f^*:W^*\rightarrow(V^*,\mathcal{J}^*)$ induces a generalized complex structure on the subspace $W^*\subseteq V^*$. This transposition means that we look at the pairs as $(\alpha,X)$ instead of $(X,\alpha)$, which has the effect of interchanging $\pi$ with $\sigma$ and making the push-forward with respect to the mapping $f_*$ become the pullback with respect to the mapping $f^*$.
If we transpose the existence conditions of the induced structure, we get the following results.
\begin{prop}\label{condproject} {\rm1.} The image $W$ of the surjection $f_*:(V,\mathcal{J})\rightarrow W$ has the projected generalized complex structure $\tilde{\mathcal{J}}$ in the sense of Definition {\rm\ref{projstr}} iff
\begin{equation}\label{goodsur} W^*\cap\flat_\sigma(ker\,f_*)=0,
\,A^*(W^*)\subseteq W^*+im\,\flat_\sigma,\end{equation} equivalently,
\begin{equation}\label{auxsur} V=ker\,f_*
+ann(\flat_\sigma(ker\,f_*)),\,A^*(W^*)\subseteq W^*\oplus\flat_\sigma(ker\,f_*).\end{equation}
{\rm2.} The existence of the projected structure is also characterized by each of the following equivalent conditions {\rm(a*)} $B=B\cap(\mathcal{J}B)+ker\,f_*$, {\rm(b*)} $\mathcal{J}B\subseteq B+\mathcal{J}(ker\,f_*)$, {\rm(c*)} $(\mathcal{J}B)^{\perp_g}\cap B\subseteq ker\,f_*$.
\end{prop}
\begin{prop}\label{exprJproj} Let $\tilde{\mathcal{J}}$ be  the (implicitly assumed to exist) projected generalized complex structure of $W$ defined by the surjection $f_*:(V,\mathcal{J})\rightarrow W$. For any pair $\mathcal{X}=(X,\alpha)\in B$, consider the following decompositions given {\rm(\ref{auxsur})} and {\rm(b*)} of Proposition {\rm\ref{condproject}}
$$A^*\alpha=\flat_\sigma Z+\nu,\,X=Z'+U,\,\mathcal{J}\mathcal{X}=\mathcal{X}'+\mathcal{J}\mathcal{X}_0,$$
where $Z,Z'\in ker\,f_*,\, U\in ann(\flat_\sigma(ker\,f_*),\,\nu\in W^*$ and $\mathcal{X}'\in B, \mathcal{X}_0\in ker\,f_*$. Then,
\begin{equation}\label{Jproject} \tilde{\mathcal{J}}(q(X,\alpha))=q(\mathcal{J}((U+Z,\alpha))\;(X,\alpha)\in B)
\end{equation} and also
$\tilde{\mathcal{J}}(q\mathcal{X})=q(\mathcal{X}') =q(\mathcal{J}(\mathcal{X}-\mathcal{X}_0))$.
\end{prop}

The first condition (\ref{goodsur}) is equivalent to $\overrightarrow{f}_*(graph\,\flat_\sigma)=graph\,\flat_{\tilde{\sigma}}$ for some $2$-form $\tilde{\sigma}\in\wedge^2W^*$ ($\flat_\sigma:V\rightarrow V^*,\flat_{\tilde{\sigma}}:W\rightarrow W^*$). Then, formula (\ref{auxsur}) implies a decomposition $V=Q\oplus(ker\,f_*)$ with $Q\perp_\sigma(ker\,f_*)$, whence, we deduce $f^*\tilde{\sigma}=\sigma|_Q$. Furthermore, using (\ref{Jproject}) we obtain the tensors of the projected structure $\tilde{\mathcal{J}}$, namely $\tilde{\sigma}$ deduced above and, with the identification $W\approx Q$, $\tilde{A}=pr_{Q}\circ A, \sharp_{\tilde{\pi}}=pr_{Q}\circ\sharp_\pi$.

The conditions for a linear mapping $f_*:(V,\mathcal{J})\rightarrow W$ to be {\it forward regular} in the sense that $\overrightarrow{f}_*L\cap\overrightarrow{f}_*\bar{L}=(ker\,f^*)^c$
are those ensuring that $f^*:W^*\rightarrow(V^*,\mathcal{J}^*)$ is backward regular. Thus, they will be obtained by transferring (\ref{1ptquasi}) to $f^*$. The result is
$$ im\,f^*\cap\flat_\sigma (ker\,f_*)=0,\; A^*(im\,f^*)
\subseteq im\,f^*+im\,\flat_\sigma.$$

The case of a smooth manifold submersion $f:(N,\mathcal{J})\rightarrow M$ cannot be handled in the same way as the case of an immersion since the inverse image of $f^{-1}(x)$ ($x\in M$) is not just one point of $N$. Instead, we may proceed as follows. We will say that the structure $\mathcal{J}$ is {\it projectable} if, $\forall(X,f^*\eta)\in\Gamma\mathbf{T}N$, with $\eta\in\Omega^1(M)$ and $X\in\Gamma_{pr}N$, where $\Gamma_{pr}$ denotes the space of projectable vector fields, one has $\mathcal{J}(X,f^*\eta)=(X',f^*\eta')$ where $\eta'\in\Omega^1(M)$ and $X'\in\Gamma_{pr}N$. It is easy to see that this condition is equivalent to the projectability of the tensors $A,\sigma,\pi$, i.e., the existence of $f$-related tensors $\tilde{A},\tilde{\sigma},\tilde{\pi}$ on $M$, and there exists a projected structure $\tilde{\mathcal{J}}$ such that, $\forall y\in M$, $\tilde{L}_y=\overrightarrow{f}_{*x}L_x$ whenever $y=f(x)$.\\

Now, we shall consider a different issue. In the classical case, the induced structure is characterized by the fact that the immersion is a complex (holomorphic) mapping. The literature on generalized complex geometry contains several attempts to extend this notion \cite{{Cr},{Mil},{OP}}, which lead to rather restrictive situations and a really good notion of generalized complex mapping may not exist.

We shall refer to the linear case. Then, the case of smooth mappings $f:N\rightarrow M$ may be handled by applying the linear case definitions to the differential of $f$.

The linear mapping $f_*:(V,\mathcal{J}_V)\rightarrow (W,\mathcal{J}_W)$ is generalized complex in the sense of \cite{Cr} iff $A_W\circ f_*=f_*\circ A_V,f_*\pi_V=\pi_W,f^*\sigma_W=\sigma_V$. These conditions may not hold for induced structures.

The linear mapping $f_*:(V,\mathcal{J}_V)\rightarrow (W,\mathcal{J}_W)$ is generalized holomorphic in the sense of \cite{OP} if $f_*(pr_VL_V)\subseteq pr_WL_W$ and $f_*\pi_V=\pi_W$. Calculations show that $\pi_V,\pi_W$ are the Poisson structures considered in \cite{OP} and that the conditions of \cite{OP} are equivalent to
\begin{equation}\label{OPcond} \begin{array}{l} \sharp_{\pi_W}=f_*\circ\sharp_{\pi_V}\circ f^*, \, f_*(im\,\sharp_{\pi_V})\subseteq im\,\sharp_{\pi_W},\vspace*{2mm}\\
im(f_*\circ A_V-A_W\circ f_*)\subseteq im\,\sharp_{\pi_W}.
\end{array}\end{equation}
Again, these conditions may not hold for induced structures.

We do not discuss the definitions given in \cite{Mil}, leaving them to the reader. On the other hand, we propose a new definition.
\begin{defin}\label{holnou} {\rm Let $f_*:(V,\mathcal{J}_V)\rightarrow(W,\mathcal{J}_W)$ be a linear mapping. Then, $f_*$ is {\it generalized complex} if
\begin{equation}\label{genh} \overrightarrow{f}_*L_V\subseteq L_W+(ker\,f^*)^c,\,\overleftarrow{f}^*L_W\subseteq L_V+(ker\,f_*)^c.\end{equation}
}\end{defin}

By using formula (\ref{eqluiL}) in the definition of the bundles included in (\ref{genh}), we can see that $f_*$ is generalized complex in the new sense iff, $\forall X\in V,\eta\in W^*$, the following equivalence holds
\begin{equation}\label{eqgenh} \begin{array}{l}
(f^*A^*_W-A^*_Vf^*) \eta
=(f^*\flat_{\sigma_W}f_*-\flat_{\sigma_V})X\vspace*{2mm}\\
\Leftrightarrow\,(A_Wf_*-f_*A_V)X=(f_*\sharp_{\pi_V}f^*-
\sharp_{\pi_W})\eta. \end{array}\end{equation}

If $\sigma_V=0,\sigma_W=0,\pi_V=0,\pi_W=0$, condition (\ref{eqgenh}) becomes $(f^*A^*_W-A^*_Vf^*) \theta=0\Leftrightarrow (A_Wf_*-f_*A_V)X=0$. Since the first part of the equivalence holds for $\theta=0$, the second part must hold for any $X$, hence, we are exactly in the case of a classical complex mapping. By a similar argument, if $f_*$ is a generalized complex mapping and if the condition $A_W\circ f_*-f_*\circ A_V=0$ holds, then, we must also have $f^*\sigma_W=\sigma_V,f_*\pi_V=\pi_W$ and we are in the case of \cite{Cr}. In particular, the case $A_V=A_W=0$ shows that the generalized complex mappings between symplectic spaces (seen as generalized complex) are just the symplectic mappings.

If the inclusion $\iota_*:(V,\tilde{\mathcal{J}})\subseteq(W,\mathcal{J})$ is a generalized complex mapping in the sense of Definition \ref{holnou}, the second relation (\ref{genh}) shows that $V$ is a generalized complex subspace, where $\tilde{\mathcal{J}}$ is the induced structure.
Conversely, if the inclusion $\iota_*:V\subseteq(W,\mathcal{J}))$ induces a generalized complex structure $\tilde{\mathcal{J}}$, $\iota_*:(V,\tilde{\mathcal{J}})\subseteq(W,\mathcal{J})$ is a generalized complex mapping. Indeed, the second relation (\ref{genh}) obviously holds and we have to check that $\overrightarrow{\iota}_*(\overleftarrow{\iota}^*L)\subseteq L+(ker\,f^*)^c$. Since $\iota_*$ is injective, $(\iota_*X,\eta)\in \overrightarrow{\iota}_*\tilde{L}$ ($X\in V,\eta\in W^*$) iff $(X,\iota^*\eta)\in\overleftarrow{\iota}^*L$. The latter condition is equivalent to $\exists\eta^0\in (ker\,\iota^*)^c$ such that $(\iota_*X,\eta+\eta^0)=(Y,\zeta)\in L$. Accordingly, $(\iota_*X,\eta)=(Y,\zeta)+(0,-\eta^0)$ and we are done.
\section{Submanifolds with induced generalized K\"ahler structure}
In this section we discuss induced generalized almost Hermitian and generalized K\"ahler structures and we begin with the following result.
\begin{prop}\label{Gindus} Any subspace $V\subseteq(W,G)$, where $G$ is a generalized Euclidean metric, inherits a naturally induced generalized Euclidean metric.\end{prop}
\begin{proof} Let $(\gamma,\psi)$ be the corresponding pair of $G$. For any linear mapping $f_*:V\rightarrow W$, we get
$$\overleftarrow{f}^*W_\pm=\{(X,f^*\flat_{\psi\pm\gamma}f_*X)\,/\,X\in V\},$$
hence, $dim\,\overleftarrow{f}^*W_\pm=dim\,V$. Then,
$$\overleftarrow{f}^*W_+\cap \overleftarrow{f}^*W_-= \{(X,f^*\flat_{\psi\pm\gamma}f_*X)\,/\,\gamma(f_*X,f_*X)=0\}$$
and, since $\gamma$ is positive definite, we get
\begin{equation}\label{Vpmcap}\overleftarrow{f}^*W_+\cap \overleftarrow{f}^*W_-=ker\,f_*.\end{equation} Thus, if $f_*$ is an injection, we get $\mathbf{V}=V_+\oplus V_-$, where $V_+=\overleftarrow{f}^*W_+,V_-=\overleftarrow{f}^*W_-$, and the decomposition defines the induced metric. We will denote it by $\overleftarrow{f}^*G$ and it corresponds to the pair $(f^*\gamma,f^*\psi)$.
 Q.e.d. \end{proof}

Now we define and characterize the induced structures.
\begin{defin}\label{defHsubsp} {\rm Assume that (\ref{descptKgen}) defines a generalized Hermitian structure on $W$. The subspace defined by the injection $\iota_*:V\subseteq W$ is a {\it generalized Hermitian subspace} if the spaces $\overleftarrow{\iota}^*L_\pm$ define a generalized Hermitian structure on $V$. The latter will be called the {\it induced structure}.}\end{defin}
\begin{prop}\label{propbicomplex} The subspace $V$ of the generalized Hermitian space $(W,G,\mathcal{J})$, with the equivalent quadruple $(\gamma,\psi,J_\pm)$, is a generalized Hermitian subspace iff $J_\pm(V)=V$.\end{prop}
\begin{proof} For any linear mapping $f_*:V\rightarrow W$ one has
\begin{equation}\label{f*auxbK} \begin{array}{l}
\overleftarrow{f}^*L_\pm=\{(X,f^*\flat_{\psi\pm\gamma}f_*X)\,/\,X\in V^c,\,J_\pm f_*X=if_*X\}\vspace*{2mm}\\
=\{(X,\flat_{f^*\psi\pm f^*\gamma}X)\,/\,X\in V^c,\,\tau_\pm(f_*X)\in L_\pm\}\subseteq\overleftarrow{f}^*W_\pm.\end{array}\end{equation}
The definition (\ref{pullbackgeneral}) of a pullback  shows that $ker\,f_*\subseteq\overleftarrow{f}^*U$ holds for any mapping $f_*:V\rightarrow W$ and any subset $U\subseteq\mathbf{W}=W\oplus W^*$.
Accordingly, from (\ref{f*auxbK}), we get
\begin{equation}\label{capbackLpm}	 \overleftarrow{f}^*L_+\cap\overleftarrow{f}^*\bar{L}_+= \overleftarrow{f}^*L_-\cap\overleftarrow{f}^*\bar{L}_-=(ker\,f_*)^c.
\end{equation}
In particular, if $f_*$ is injective, (\ref{f*auxbK}), (\ref{capbackLpm}) yield
\begin{equation}\label{auxJpminv}
\overleftarrow{f}^*L_\pm=\tau_\pm(V^c\cap T_\pm),\,\overleftarrow{f}^*L_\pm\cap \overleftarrow{f}^*\bar{L}_\pm=0,\end{equation}
where $T_\pm$ is the $i$-eigenspace of the complex structure $J_\pm$ of $W_\pm$.

Simple calculations that use (\ref{f*auxbK}) show that $\overleftarrow{f}^*L_+\oplus\overleftarrow{f}^*L_-,
\overleftarrow{f}^*L_+\oplus\overleftarrow{f}^*\bar{L}_-$ are $g$-isotropic subspaces of $\mathbf{V}^c$ and
$\overleftarrow{f}^*L_+\oplus\overleftarrow{f}^*\bar{L}_+,
\overleftarrow{f}^*L_-\oplus\overleftarrow{f}^*\bar{L}_-$ are $g$-positive and $g$-negative, respectively. In view of these facts, if $f_*$ is injective, the subspaces $\overleftarrow{f}^*L_\pm\subseteq\mathbf{V}^c$ satisfy conditions (i), (ii) that characterize the decomposition (\ref{descptKgen}) of a generalized Hermitian space iff $dim\,\overleftarrow{f}^*L_\pm=dim\,V/2$. Equalities (\ref{auxJpminv}) show that the previous dimension condition holds iff $V^c\cap T_\pm$ are $i$-eigenspaces of complex structures of $V$, which holds iff $J_\pm(V)=V$. Q.e.d. \end{proof}

A subspace $V$ of a space $W$ endowed with two complex structures $J_\pm$ such that $J_\pm(V)=V$ will be called a {\it bicomplex subspace}.
\begin{prop}\label{bi=cJJ'} The bicomplex subspace $V$ of the linear, generalized Hermitian space $(W,G,\mathcal{J},\mathcal{J}')$ is a generalized complex subspace of both $(W,\mathcal{J})$ and $(W,\mathcal{J}')$. Moreover, the pair of induced structures $\tilde{\mathcal{J}},\tilde{\mathcal{J}}'$ coincides with the pair of the induced generalized Hermitian structure $\overleftarrow{\iota}^*L_\pm$ ($\iota_*:V\subseteq W$).\end{prop}
\begin{proof} We have to check the existence conditions of the induced structures:
\begin{equation}\label{auxcomparare} \begin{array}{l} V\cap\sharp_\pi(ann\,V)=0,\;AV\subseteq V+\sharp_\pi(W^*),\vspace*{2mm}\\ V\cap\sharp_{\pi'}(ann\,V)=0,\;A'V\subseteq V+\sharp_\pi'(W^*). \end{array}\end{equation}

Take $\lambda\in ann\,V$, equivalently, $\lambda=\flat_\gamma Z$, where $Z\in V^{\perp_\gamma}$, and assume that either $X=\sharp_\pi\lambda\in V$ or $X=\sharp_{\pi'}\lambda\in V$. Then, (\ref{clasicK}) gives $X=(1/2)(J_+\mp J_-)Z$, respectively, where, since $J_\pm(V)=V$, the right hand side belongs to $V^{\perp_\gamma}$ because of the $\gamma$-compatibility of $J_\pm$. Hence, $X=0$ and the first part of the two lines of (\ref{auxcomparare}) holds.

Furthermore, for $X\in V$, (\ref{clasicK}) gives
$$\begin{array}{l}
2AX=(J_++J_-)X+(J_+-J_-)\varphi X\vspace*{2mm}\\ =
(J_++J_-)X+(J_+-J_-)(pr_{V}\varphi X) + (J_+-J_-)(pr_{V^{\perp_\gamma}}\varphi X)\vspace*{2mm}\\ =(J_++J_-)X+(J_+-J_-)(pr_{V}\varphi X) +\frac{1}{2} \sharp_\pi(\flat_\gamma(pr_{V^{\perp_\gamma}}\varphi X)).\end{array}$$
In the case of a bicomplex subspace the first two terms belong to $V$ and we see that the second condition of the first line of (\ref{auxcomparare}) holds. The second condition in the second line of (\ref{auxcomparare}) is justified similarly.

Now, for a moment, we denote by $\dot{\mathcal{J}},\dot{\mathcal{J}}'$ the structures induced by $\mathcal{J},\mathcal{J}'$. These have the $i$-eigenspaces $\dot{L}=\overleftarrow{\iota}^*(L_+\oplus L_-),\dot{L}'=\overleftarrow{\iota}^*(L_+\oplus \bar{L}_-)$. On the other hand, the generalized complex structures included in the induced generalized Hermitian structure have the $i$-eigenspaces $\overleftarrow{\iota}^*L_+\oplus\overleftarrow{\iota}^*L_-,\linebreak
\overleftarrow{\iota}^*L_+\oplus\overleftarrow{\iota}^*\bar{L}_-$. The latter are included in $\dot{L},\dot{L}'$ and have the same dimension. Hence, $\dot{\mathcal{J}}=\tilde{\mathcal{J}},\dot{\mathcal{J}}' =\tilde{\mathcal{J}}'$.
 Q.e.d. \end{proof}
\begin{prop}\label{HcuB} $V$ is a generalized Hermitian subspace of $W$ iff $V$ is a generalized complex subspace of both $(W,\mathcal{J})$, $(W,\mathcal{J}')$ and\footnote{Theorem 8.1 of \cite{BS} asserts that the condition that follows is ensured by the first condition, but, it seems the proof has an error that we could not overcome.} $s(B\cap(\mathcal{J}B)\cap(\mathcal{J}'B)\cap(HB))=\mathbf{V}$.\end{prop}
\begin{proof} Recall that $B=V\oplus W^*\subseteq\mathbf{W}$ and $s:B\rightarrow\mathbf{V}$ is the projection onto the quotient of $B$ by $ann\,V$.
Denote $\mathcal{B}=B\cap(\mathcal{J}B)\cap(\mathcal{J}'B)\cap(HB)$. The last condition of the proposition is equivalent to $ B=\mathcal{B}+ann\,V$ and
$\mathcal{X}\in\mathcal{B}$ iff $\mathcal{X},\mathcal{J}\mathcal{X},\mathcal{J}'\mathcal{X}, H\mathcal{X}\in B$. Accordingly, and since $L_\pm=W_\pm\cap L$ are intersections of eigenspaces, we get $$B\cap L_+\subseteq\mathcal{B},B\cap L_-\subseteq\mathcal{B},B\cap \bar{L}_+\subseteq\mathcal{B},B\cap \bar{L}_-\subseteq\mathcal{B}$$ and also, $\forall\mathcal{X}\in\mathcal{B}$,  $$pr_{L_+}\mathcal{X}=\frac{1}{4}(Id+H)(Id-i\mathcal{J})\mathcal{X} =\frac{1}{4}(\mathcal{X}-i\mathcal{J}\mathcal{X}-i\mathcal{J}'\mathcal{X} +H\mathcal{X}) \in B\cap L_+.$$ Similar results hold for the projections of $\mathcal{X}$ on $L_-,\bar{L}_+,\bar{L}_-$. Using these facts, we see that
\begin{equation}\label{Bcucap}\mathcal{B}=(B\cap L_+)\oplus(B\cap L_-)\oplus(B\cap \bar{L}_+)\oplus(B\cap \bar{L}_-).\end{equation}

For any subset $U\subseteq\mathbf{W}$ we may identify $\overleftarrow{\iota}^*U=s(B\cap U)$. Accordingly, if we apply $s$ to (\ref{Bcucap}) and use
$\overleftarrow{\iota}^*L\cap\overleftarrow{\iota}^*\bar{L}=0, \overleftarrow{\iota}^*L'\cap\overleftarrow{\iota}^*\bar{L}'=0$, which follow from the fact that $W$ is a generalized complex subspace for both $\mathcal{J},\mathcal{J}'$, we get
\begin{equation}\label{Bronddupacap} s(\mathcal{B})=\overleftarrow{\iota}^*L_+ \oplus\overleftarrow{\iota}^*L_-\oplus\overleftarrow{\iota}^*\bar{L}_+ \oplus\overleftarrow{\iota}^*\bar{L}_-.\end{equation}
This result implies the conclusion of the proposition.	Q.e.d. \end{proof}
\begin{rem}\label{obsGindus} {\rm By a similar procedure, one can prove that, if $H$ is the paracomplex structure that corresponds to an arbitrary generalized Euclidean metric $G$, then, $B=B\cap(HB)+ann\,V$.}\end{rem}
\begin{example}\label{exsHg} {\rm For any space $(W,\gamma,J_\pm)$, put $S^c_+=(T_+\cap T_-)\oplus(\bar{T}_+\cap \bar{T}_-)$, $S^c_-=(T_+\cap \bar{T}_-)\oplus(\bar{T}_+\cap T_-)$, $U=(S_+\oplus S_-)^{\perp_\gamma}$ (obviously, $S_+\cap S_-=0$) \cite{OP}. By decomposing vectors as sums of eigenvectors of $J_\pm$, it follows that $S_\pm$ are bicomplex subspaces and the $\gamma$-compatibility of $J_\pm$ implies that the same holds for $U$.

Let $W$ be a generalized Hermitian space such that $J_\pm$ commute \cite{AG}. Then, if $S$ is a $J_+$-invariant subspace of $W$ the subspace $V=S\cap(J_-S)$ is bicomplex.

If $V\subseteq(W,G,\mathcal{J},\mathcal{J}')$ is $\mathcal{J}$-totally invariant, $\mathcal{J}B=B$ and $V$ has the induced structure (Example \ref{exindcstr}). Moreover, $V$ is a bicomplex subspace \cite{Goto}. Indeed, we have, $\mathcal{J}'B=H(\mathcal{J}B)=HB$ and using Remark \ref{obsGindus} we get $B=B\cap(HB)+ann\,V=B\cap(\mathcal{J}'B)+ann\,V$, which shows that $V$ is also a $\mathcal{J}'$-generalized complex subspace. The last condition of Proposition \ref{HcuB} holds too, because of Remark \ref{obsGindus} and since $B\cap(\mathcal{J}B)\cap(\mathcal{J}'B)\cap(HB)=B\cap(HB)$.

Finally, it is easy to check that the $B$-field transformation of the generalized Hermitian structure defined by the quadruple $(\gamma,\psi,J_\pm)$ yields the generalized Hermitian structure defined by the quadruple $(\gamma,\psi-B,J_\pm)$. Since $J_\pm$ remain unchanged, we see that a bicomplex subspace remains bicomplex after a $B$-field transformation.}\end{example}
\begin{prop}\label{sbcuJpm} A subspace $V$ of a generalized Hermitian space $W$ is generalized complex with respect to the two structures $\mathcal{J},\mathcal{J}'$ of $W$ iff it satisfies the conditions
\begin{equation}\label{sbJpm1}V\cap[(J_+-J_-)(V^{\perp_\gamma})]=0,\, V\cap[(J_++J_-)(V^{\perp_\gamma})]=0\end{equation}
and one of each pair of conditions
\begin{equation}\label{sbJpm2} J_\pm(V)\subseteq V+(J_+-J_-)(W), J_\pm(V)\subseteq V+(J_++J_-)(W).\end{equation}
\end{prop}
\begin{proof} The expression (\ref{clasicK}) of $\pi,\pi'$ yields $$\sharp_\pi(ann\,V)=(J_+-J_-)(V^{\perp_\gamma}),\, \sharp_{\pi'}(ann\,V)=(J_++J_-)(V^{\perp_\gamma}).$$ Accordingly, the first conditions in the two lines of (\ref{auxcomparare}) translate into (\ref{sbJpm1}) and their meaning is that $V$ is a Poisson-Dirac subspace for the bivectors $\pi,\pi'$. If this happens, formula (\ref{JApi}) allows us to check the equivalence of the second and fourth condition (\ref{auxcomparare}) with each pair in (\ref{sbJpm2}).
 Q.e.d. \end{proof}
\begin{rem}\label{obsJpmj} {\rm Formulas (\ref{clasicK}) shows that $\pi,\pi'$ are non degenerate iff $J_+\pm J_-$ are non degenerate. Then, (\ref{sbJpm1}) holds iff $V$ is a symplectic subspace for both structures $\pi^{-1},\pi^{'-1}$. Moreover, in this case, all the conditions (\ref{sbJpm2}) necessarily hold and $V$ is a generalized subcomplex space of both $(W,\mathcal{J})$ and $(W,\mathcal{J}')$. Notice that
$$J_+-J_-=2i(pr_{T_+}-pr_{T_-}),\,J_++J_-=2i(pr_{T_+}-pr_{\bar{T}_-}),$$
where the projections are defined by the decompositions of $W^c$ as sums of eigenspaces of $J_\pm$, $W^c=T_+\oplus \bar{T}_+,\,W^c=T_-\oplus \bar{T}_-$.
Hence, $J_+\pm J_-$ have kernel zero and image $W$ (are isomorphisms) iff $T_+\cap T_-=0, T_+\cap \bar{T}_-=0$. Notice also that, if $\psi=0$, (\ref{clasicK}) imply
$A=(J_++J_-)/2,\,A'=(J_+-J_-)/2$.}\end{rem}

Now, let $f_*:(V,G,\mathcal{J})\rightarrow W$ be a surjection defined on a generalized Hermitian space with the generalized complex structure $\mathcal{J}$ and the metric $G$ equivalent to the paracomplex structure $H$. We will establish the conditions for the push-forward spaces $\overrightarrow{f}_*L_\pm$ to define a {\it projected} generalized Hermitian structure on $W$.
\begin{prop}\label{propprojHerm} The projected Hermitian structure exists iff $ker\,f_*$ is a bicomplex subspace of $(V,J_\pm)$.\end{prop}
\begin{proof} Obviously, the projected structure exists iff the injection $f^*:W^*\subseteq(V^*,G^*,\mathcal{J}^*)$, where the star always denotes the dual object, makes $W^*$ into a generalized Hermitian subspace of $V^*$. By Proposition \ref{propbicomplex},
$W^*$ has an induced generalized Hermitian structure iff it is invariant by the complex structures $J^*_\pm$ and this condition is equivalent to $J_\pm(ker\,f_*)=ker\,f_*$. Q.e.d. \end{proof}

We give a few more details about the projected structure. Assume that the metric $G$ of $\mathbf{V}$ corresponds to the quadruple $(\gamma,\psi,J_\pm)$ and put $Q=(ker\,f_*)^{\perp_\gamma}$. We may identify $Q$ with $W$ by the isomorphism $f_*|_Q:Q\rightarrow W$ and $Q^*$ with $W^*=ann(ker\,f_*)$. In view of the $\gamma$-compatibility of $J_\pm$, the $J_\pm$-invariance of $ker\,f_*$ is equivalent to $J_\pm(Q)=Q$. Using again Proposition \ref{propbicomplex}, we get the induced generalized Hermitian structure $(\tilde{G},\tilde{\mathcal{J}})$ on $Q$ defined by $(\gamma|_Q,\psi|_Q,J_\pm|_Q)$. On the other hand, the projected structure of $W$ is the dual of the structure defined on $W^*$ by the quadruple $(\tilde{\gamma}^*,\tilde{\psi}^*,J^*_\pm)$ of objects induced by the inclusion $f^*$. The structures of $Q,W$ coincide under the natural identification $Q\approx W$.

Now, we shall refer to submanifolds $\iota:N\hookrightarrow(M,G,\mathcal{J},\mathcal{J}')$, where $M$ is a generalized almost Hermitian manifold with the structure $(G,\mathcal{J},\mathcal{J}')$.
\begin{defin}\label{defsubKgen} {\rm $N$ is a generalized almost Hermitian submanifold if it has the	 induced structure at every point $x\in N$ and the differentials $\iota_{*x}$ pull back the subbundles $L_\pm$ of $\mathbf{T}^cM$ to smooth subbundles of $\mathbf{T}^cN$.}\end{defin}
\begin{prop}\label{propsubK1} Let $\iota:N\hookrightarrow(M,G,\mathcal{J},\mathcal{J}')$ be a generalized almost Hermitian submanifold of a generalized K\"ahler manifold $M$. Then, $N$ is a generalized K\"ahler manifold too.\end{prop}
\begin{proof} By definition, $M$ is generalized K\"ahler if the generalized structures $\mathcal{J},\mathcal{J}'$ are integrable and Proposition \ref{sbvardif} tells us that, if the induced structures of $N$ exist, they are integrable too. Q.e.d.\end{proof}

From the results obtained in the linear case, it follows that $N$ is a generalized almost Hermitian submanifold iff it is a {\it bicomplex submanifold} of $M$, i.e., $N$ is invariant by the two structures $J_\pm$. Particularly, formula (\ref{f*auxbK}) shows that $\overleftarrow{\iota}^*L_\pm$ are smooth subbundles of $\mathbf{T}^cN$.
This leads to the following result.
\begin{prop}\label{varbix} A submanifold $\iota:N\hookrightarrow M$ of a generalized K\"ahler manifold $M$ is a generalized K\"ahler submanifold iff $N$ is a complex submanifold of the two complex manifolds $(M,J_\pm)$. \end{prop}
\begin{proof} We already know that $N$ must be $J_\pm$-invariant and, of course, if $J_\pm$ are integrable, the induced structures are integrable too. The required supplementary condition for $N$ to be generalized K\"ahler is (\ref{KalercuLC}) on $N$ and it follows by taking the $\gamma$-orthogonal projection of the equality (\ref{KalercuLC}) for $M$ onto the tangent spaces of $N$. The projection sends the left hand side of (\ref{KalercuLC}) to the covariant derivative with respect to Levi-Civita connection of the induced metric $\iota^*\gamma$ (look at the Gauss-Weingarten equations \cite{KN}) and the right hand side to the similar expression for the induced form $\iota^*\psi$. Q.e.d. \end{proof}

In order to get a manifold version of Proposition \ref{propprojHerm}, we proceed as in the generalized complex case. We consider a submersion $f:(N,\mathcal{J},\mathcal{J}',G)\rightarrow M$ where the structures $\mathcal{J},\mathcal{J}'$ define a generalized almost Hermitian structure of $N$ and are projectable. Then, $H=-\mathcal{J}\circ\mathcal{J}'$ is projectable too, i.e., it sends projectable cross sections to projectable cross sections. In this case, we obviously get a projection of the generalized almost Hermitian structure of $N$ onto $M$. Formulas (\ref{JApi}) and (\ref{clasicK}) show that the projectability of $\mathcal{J},\mathcal{J}',H$ is equivalent to the projectability of the corresponding classical structures $(\gamma,\psi,J_\pm)$. The latter condition means that $\gamma|_{(ker\,f_*)^{\perp_\gamma}},\psi|_{(ker\,f_*)^{\perp_\gamma}}$ are pullbacks of a metric and of a $2$-form on $M$ and $J_\pm|_{(ker\,f_*)^{\perp_\gamma}}$ sends projectable vector fields to projectable vector fields. Proposition \ref{propprojHerm} tells us that a submersion projects an almost generalized Hermitian structure iff $J_\pm$ are projectable and preserve the kernel $ker\,f_*$. Moreover, if the structure of $N$ is integrable, i.e., $J_\pm$ are integrable and (\ref{KalercuLC}) holds, the projected structure has the same properties. Indeed, the Lie bracket of projectable vector fields is projectable to the corresponding Lie bracket and the Levi-Civita connection of $G$ projects to that of the projected metric (for details, a text on foliations and metrics, e.g., \cite{Mol}, should be consulted). The conclusion is that a submersion sends a projectable generalized K\"ahler structure to a generalized K\"ahler structure.

\noindent{\small Department of
Mathematics, University of Haifa, Israel,
vaisman@math.haifa.ac.il}
\end{document}